\documentclass[12pt,reqno]{amsart}
\usepackage{amscd,amsmath,amsthm,amssymb}
\usepackage{color}
\usepackage{pstricks}
\usepackage{stmaryrd}
\usepackage{tikz}
\usepackage{url}

\usepackage{latexsym}
\usepackage{amsfonts,amsmath,mathtools}
\usepackage{graphics}
\usepackage{float}
\usepackage{enumitem}

\usepackage{booktabs} 
\usepackage{colortbl}
\usepackage{lipsum}

\newpsstyle{fatline}{linewidth=1.5pt}
\newpsstyle{fyp}{fillstyle=solid,fillcolor=verylight}
\definecolor{verylight}{gray}{0.97}
\definecolor{light}{gray}{0.9}
\definecolor{medium}{gray}{0.85}
\definecolor{dark}{gray}{0.6}

\def\NZQ{\mathbb}               
\def\NN{{\NZQ N}}

\def\ZZ{{\NZQ Z}}
\def\RR{{\NZQ R}}

\def\frk{\mathfrak}               

\def\pp{{\frk p}}

\def\mm{{\frk m}}

\def\G{{\mathcal G}}

\def\pd{\textup{proj}\phantom{.}\!\textup{dim}}

\def\opn#1#2{\def#1{\operatorname{#2}}} 
\opn\chara{char} \opn\length{\ell} \opn\pd{pd} \opn\rk{rk}
\opn\projdim{proj\,dim} \opn\injdim{inj\,dim} \opn\rank{rank}
\opn\depth{depth} \opn\grade{grade} \opn\height{height}
\opn\embdim{emb\,dim} \opn\codim{codim}

\opn\Tr{Tr} \opn\bigrank{big\,rank}
\opn\superheight{superheight}\opn\lcm{lcm}
\opn\trdeg{tr\,deg}
	\opn\reg{reg} \opn\lreg{lreg} \opn\ini{in} \opn\lpd{lpd}
	\opn\size{size} \opn\sdepth{sdepth}
	\opn\link{link}\opn\fdepth{fdepth}\opn\lex{lex}
	\opn\tr{tr}
	\opn\type{type}
	\opn\gap{gap}
	\opn\diam{diam}
	\opn\Mod{Mod}
	\opn\div{div} \opn\Div{Div} \opn\cl{cl} \opn\Cl{Cl}
	\opn\Spec{Spec} \opn\Supp{Supp} \opn\supp{supp} \opn\Sing{Sing}
	\opn\Ass{Ass} \opn\Min{Min}\opn\Mon{Mon}
	\opn\Ann{Ann} \opn\Rad{Rad} \opn\Soc{Soc}
	\opn\Im{Im} \opn\Ker{Ker} \opn\Coker{Coker} \opn\Am{Am}
	\opn\Hom{Hom} \opn\Tor{Tor} \opn\Ext{Ext} \opn\End{End}
	\opn\Aut{Aut} \opn\id{id}
	
	\opn\nat{nat}
	\opn\pff{pf}
	\opn\Pf{Pf} \opn\GL{GL} \opn\SL{SL} \opn\mod{mod} \opn\ord{ord}
	\opn\Gin{Gin} \opn\Hilb{Hilb}\opn\sort{sort}
	\opn\PF{PF}\opn\Ap{Ap}
	\opn\dist{dist}
	\opn\aff{aff}
	\opn\relint{relint} \opn\st{st}
	\opn\lk{lk} \opn\cn{cn} \opn\core{core} \opn\vol{vol}  \opn\inp{inp} \opn\nilpot{nilpot}
	\opn\link{link} \opn\star{star}\opn\lex{lex}\opn\set{set}
	\opn\width{wd}
	\opn\Fr{F}
	\opn\QF{QF}
	\opn\G{G}
	\opn\type{type}\opn\res{res}
	\opn\conv{conv}
	\opn\sr{sr}
	\opn\gr{gr}
	
	\def\pot#1#2{#1[\kern-0.28ex[#2]\kern-0.28ex]}

	\opn\dirlim{\underrightarrow{\lim}}
	\opn\inivlim{\underleftarrow{\lim}}
	\def\Implies{\ifmmode\Longrightarrow \else
		\unskip${}\Longrightarrow{}$\ignorespaces\fi}
	\def\implies{\ifmmode\Rightarrow \else
		\unskip${}\Rightarrow{}$\ignorespaces\fi}
	\def\iff{\ifmmode\Longleftrightarrow \else
		\unskip${}\Longleftrightarrow{}$\ignorespaces\fi}

	\let\:=\colon
	\newtheorem{Theorem}{Theorem}[section]
	\newtheorem{Lemma}[Theorem]{Lemma}
	\newtheorem{Corollary}[Theorem]{Corollary}
	\newtheorem{Proposition}[Theorem]{Proposition}
	\newtheorem{Remark}[Theorem]{Remark}

	\newtheorem{Problem}[Theorem]{Problem}
	\newtheorem{Conjecture}[Theorem]{Conjecture}
	\newtheorem{Question}[Theorem]{Question}
	
	\let\epsilon\varepsilon
	\let\kappa=\varkappa
	\def\qed{\ifhmode\textqed\fi
		\ifmmode\ifinner\hfill\quad\qedsymbol\else\dispqed\fi\fi}
	\def\textqed{\unskip\nobreak\penalty50
		\hskip2em\hbox{}\nobreak\hfill\qedsymbol
		\parfillskip=0pt \finalhyphendemerits=0}
	\def\dispqed{\rlap{\qquad\qedsymbol}}
	
	\opn\dis{dis}
	\def\pnt{{\raise0.5mm\hbox{\large\bf.}}}
	
	\opn\Lex{Lex}
	\opn\Max{Max}
	\opn\Shad{Shad}
	\opn\astab{astab}
	\def\p{\mathfrak{p}}
	\def\q{\mathfrak{q}}
	\def\m{\mathfrak{m}}
	\def\bideg{\textup{bideg}}
	\opn\v{v}

\begin{document}
		
		\title{Asymptotic behaviour of integer programming and the $\textup{v}$-function of a graded filtration}	
		\author{Antonino Ficarra, Emanuele Sgroi}
		
		\address{Antonino Ficarra, Department of mathematics and computer sciences, physics and earth sciences, University of Messina, Viale Ferdinando Stagno d'Alcontres 31, 98166 Messina, Italy}
		\email{antficarra@unime.it}
		
		\address{Emanuele Sgroi, Department of mathematics and computer sciences, physics and earth sciences, University of Messina, Viale Ferdinando Stagno d'Alcontres 31, 98166 Messina, Italy}
		\email{emasgroi@unime.it}
		
		\thanks{
		}
		
		\subjclass[2020]{Primary 13F20; Secondary 13F55, 05C70, 05E40.}
		
		\keywords{$\textup{v}$-number, primary decomposition, associated primes, monomial ideals}
		
		\maketitle
		
		\begin{abstract}
			The $\text{v}$-function of a graded filtration $\mathcal{I}=\{I_{[k]}\}_{k\ge0}$ is introduced. Under the assumption that $\mathcal{I}$ is Noetherian, we prove that the $\text{v}$-function $\text{v}(I_{[k]})$ is an eventually quasi-linear function. This result applies to several situations, including ordinary powers, and integral closures of ordinary powers, among others. As another application, we investigate the asymptotic behaviour of certain integer programming problems. Finally, we present the \textit{Macaulay2} package $\texttt{VNumber}$.
		\end{abstract}

		\section*{Introduction}
		
		Let $R$ be a finitely generated $\NN$-graded domain, and let $\mm=\bigoplus_{k>0}R_k$. Suppose that $R/\mm\cong R_0$ is an infinite field. Let $I\subset R$ be a homogeneous ideal. Then, for any associated prime $\p\in\Ass(I)$, there exists a homogeneous element $f\in R$ such that $(I:f)=\p$. Following \cite{CSTVV20}, we define the $\v_\p$-number of $I$ to be the least degree of such an element. Whereas, the $\v$-number of $I$ is defined as $\v(I)=\min_{\p\in\Ass(I)}\v_\p(I)$. See also \cite{BM23,BMS24,F2023,FM,GRV21,S2023,SS20,VS24}.\smallskip
		
		In \cite{FS2}, if $R$ is the standard graded polynomial ring $S$ over a field $K$ and $I$ is a graded ideal of $S$, we proved that $\v(I^k)$ is a linear function for $k\gg0$, and $\lim_{k\rightarrow\infty}\v(I^k)/k$ is equal to the initial degree of $I$. This result was shown independently by Conca \cite{Conca23} in the more general frame of any graded ideal $I\subset R$. Further generalizations and refinements of these results were considered in \cite{Fior24,Ghosh24,KS23}.\smallskip
		
		In this paper, we generalize the results in \cite{Conca23,FS2,Fior24,Ghosh24,KS23}. Let $\mathcal{I}=\{I_{[k]}\}_{k\ge0}$ be a Noetherian graded filtration of $R$ (see Section \ref{Sec1-FS24} for the precise definitions). In our main Theorem \ref{Thm:FS24}, we prove that the function $\v(I_{[k]})$ is a quasi-linear function in $k$ for all $k\gg0$. This result applies to several situations. For instance, if $I_{[k]}=I^k$, $I_{[k]}=\overline{I^k}$ or $I_{[k]}=I^{(k)}$ for all $k\ge0$, where $I\subset R$ is a graded ideal, see also \cite{VS24} and \cite[Example 5.2]{DMNB23} for other classes of graded filtrations. In the case of a monomial ideal $I$ in the polynomial ring $S$, the computation of the functions $\v_\p(I^k)$ and $\v_\p(\overline{I^k})$ amounts to solve certain asymptotic integer programming problems and determine their optimal solutions. We explore these problems in Section \ref{Sec3-FS24}. Let $\Ass^\infty(\mathcal{I})$ be the set of those primes $\p\in\textup{Spec}(R)$ such that $\p\in\Ass(I_{[k]})$ for infinitely many $k$. In Section \ref{Sec2-FS24}, using Koszul homology and localization, we characterize the set $\Ass^\infty(\mathcal{I})$. Section \ref{Sec4-FS24} contains a self contained introduction to the \textit{Macaulay2} \cite{GDS} package \texttt{VNumber} \cite{FSPack} and its main features. We end the paper in Section \ref{Sec5-FS24} with a series of natural questions.

		\section{The $\v$-function of a graded filtration}\label{Sec1-FS24}
		
		Let $R$ be a finitely generated $\NN$-graded domain, and let $\mm=\bigoplus_{k>0}R_k$. Suppose that $R/\mm\cong R_0$ is an infinite field. A \textit{graded filtration} of $R$ is a family $\mathcal{I}=\{I_{[k]}\}_{k\ge0}$ of homogeneous ideals of $R$ satisfying:
		\begin{enumerate}
			\item[(i)] $I_{[0]}=R$,
			\item[(ii)] $I_{[k+1]}\subseteq I_{[k]}$ for all $k\ge0$,
			\item[(iii)] and $I_{[k]}I_{[\ell]}\subseteq I_{[k+\ell]}$ for all $k,\ell\ge0$.
		\end{enumerate}
		
		We use the notation $I_{[k]}$ to not confuse the ideal $I_{[k]}$ with the $k$th homogeneous component $I_k$ of a homogeneous ideal $I\subset R$.
		
		We say that $\mathcal{I}$ is \textit{Noetherian} if the Rees algebra $\mathcal{R(I)}=\bigoplus_{k\ge0}I_{[k]}$ is Noetherian.
		
		\begin{Lemma}\label{Lem:AssInftyFiltr}
			Let $\mathcal{I}=\{I_{[k]}\}_{k\ge0}$ be a Noetherian graded filtration of $R$. Then, there exists an integer $c>0$ such that $$\Ass(I_{[ck+j]})=\Ass(I_{[c(k+1)+j]})\ \ \ \textit{for all}\ k\gg0,$$ for all $j=0,\dots,c-1$.
		\end{Lemma}
		\begin{proof}
			Since $\mathcal{R(I)}$ is a graded Noetherian ring, there exists $c>0$ such that the Veronese subalgebra $$\mathcal{R(I)}^{(c)}=\bigoplus_{k\ge0}I_{[ck]}$$ is standard graded (see \cite{R79} or \cite[Theorem 2.1]{HHT2007}). For each integer $j\in\{0,\dots,c-1\}$, $\mathcal{R(I)}_{(j)}=\bigoplus_{k\ge0}I_{[ck+j]}$ is a finitely generated $\mathcal{R(I)}^{(c)}$-module. Hence \cite[Lemma 2.3]{Ahn95} (which is the module version of \cite[Proposition 2]{MAE79}) guarantees that $$\Ass(I_{[ck+j]})=\Ass(I_{[c(k+1)+j]}),$$ for all $k\gg0$ and all $j=0,\dots,c-1$.
		\end{proof}
		
		We denote by $\Ass^\infty(\mathcal{I})$ the set of those primes $\p\in\textup{Spec}(R)$ such that $\p\in\Ass(I_{[k]})$ for infinitely many $k$. Using the notation in the above lemma, $\p\in\Ass^\infty(\mathcal{I})$ if and only if $\p\in\Ass(I_{[ck+j]})$ for some $j\in\{0,\dots,c-1\}$ and all $k\gg0$.\medskip
		
		\begin{Remark}
			\rm If the Rees algebra $\mathcal{R(I)}$ is standard graded, then in Lemma \ref{Lem:AssInftyFiltr} we have $c=1$ and we obtain that $\Ass(I_{[k]})=\Ass(I_{[k+1]})$ for all $k\gg0$.
		\end{Remark}
		
		A numerical function $f:\ZZ_{\ge0}\rightarrow\ZZ_{\ge0}$ is called a \textit{quasi-linear function} (\textit{of period $c$}) if there exist a positive integer $c$ and linear functions $f_i(k)=a_ik+b_i$, $i=0,\dots,c-1$, such that $f(k)=f_i(k)$ if $k\equiv i$ (mod $c$). The aim of this section is to prove:
		
		\begin{Theorem}\label{Thm:FS24}
			Let $\mathcal{I}=\{I_{[k]}\}_{k\ge0}$ be a Noetherian graded filtration of $R$. Then $\v(I_{[k]})$ is a quasi-linear function in $k$ for all $k\gg0$.
		\end{Theorem}
		
		We call the function $\v(I_{[k]})$ the \textit{$\v$-function} of the graded filtration $\mathcal{I}$.
		
		In order to prove the theorem, we need some auxiliary lemmata. We begin with the following result of Conca \cite[Lemma 1.2]{Conca23}. Let $M$ be a graded $R$-module. We denote by $\alpha(M)=\min\{d:M_d\ne0\}$ and $\omega(M)=\max\{d:(M/\m M)_d\ne0\}$ the \textit{initial degree} and the \textit{final degree} of $M$, respectively. 
		
		\begin{Lemma}\label{Lem:Conca}
			Let $I\subset R$ be a graded ideal, and let $\p\in\Ass(I)$ be an associated prime of $I$. Let $X_\p=\{\p_1\in\Ass(I):\p\subsetneq\p_1\}$ and $\q$ be the product of the elements in $X_\p$ if $X_\p$ is non-empty, otherwise let $\q=R$. Then,
			$$
			\v_\p(I)\ =\ \alpha\Big(\frac{(I:\p)}{I:(\p+\q^\infty)}\Big).
			$$
			Here $I:(\p+\q^\infty)=\bigcup_{k\ge0}(I:(\p+\q^k))=(I:\p)\cap(\bigcup_{k\ge0}I:\q^k)$.
		\end{Lemma}
		
		A crucial property which we need for the proof of Theorem \ref{Thm:FS24} is an analogue of the so-called \textit{Ratliff property} which assures that we have $$(I^{k+1}:I)=I^k$$ for all $k\gg0$, if $I$ is an ideal of a Noetherian domain $R$ \cite[Corollary 4.2]{R79}. This property was pivotal in the proofs of \cite[Theorem 1.1]{Conca23} and \cite[Theorem 3.1]{FS2}. The following result is probably well-known in the context of graded filtrations, but due to the lack of a suitable reference, we provide a proof.
		\begin{Lemma}\label{Lem:FS24-2}
			Let $\mathcal{I}=\{I_{[k]}\}_{k\ge0}$ be a Noetherian graded filtration of $R$. Let $c>0$ be the smallest integer such that $\mathcal{R}(\mathcal{I})^{(c)}$ is standard graded. Then
			$$
			I_{[c(k+1)+j]}:I_{[c]}\ =\ I_{[ck+j]}
			$$
			for all $k\gg0$ and all $j\in\{0,\dots,c-1\}$.
		\end{Lemma}
		
		This result is a consequence of the following more general statement.
		\begin{Proposition}\label{Prop:known}
			Let $(R,\m,K)$ be either a Noetherian local ring or a finitely generated graded $K$-algebra. Let $M$ be a finitely generated $R$-module and let $I\subset R$ be an ideal. Suppose that $K$ is an infinite field and that $I$ contains a non-zero divisor on $M$. Then,
			$$
			I^{k+1}M:_M I\ =\ I^kM
			$$
			for all $k\gg0$.
		\end{Proposition}
		\begin{proof}
			The result is probably well-known to the experts, but since we could not find a precise reference in the literature, we provide a short proof. We prove it in the case that $R$ is a Noetherian local ring, the case when $R$ is a finitely generated graded $K$-algebra is analogous.
			
			By assumption, $I$ contains a non-zero divisor on $M$. Thus, \cite[Corollary 8.5.9]{HS2006} guarantees that there exists $x\in I$ which is both a non-zero divisor on $M$ and is a superficial element of $I$ with respect to $M$. Next, \cite[Lemma 8.5.3]{HS2006} implies that $I^{k+1}M:_M x=I^kM$ for all $k\gg0$. From $x\in I$, we obtain that $$I^{k+1}M:_MI\ \subseteq\ I^{k+1}M:_Mx\ =\ I^kM$$ for all $k\gg0$. Since the opposite inclusion trivially holds for all $k$, we conclude that $I^{k+1}M:_MI=I^kM$ for all $k\gg0$.
		\end{proof}
		We are now ready to prove Lemma \ref{Lem:FS24-2}.
		\begin{proof}[Proof of Lemma \ref{Lem:FS24-2}]
			Let $c>0$ be the smallest integer such that $\mathcal{R}(\mathcal{I})^{(c)}$ is standard graded. Hence $I_{[ck]}=I_{[c]}^k$ for all $k\ge0$. From the proof of \cite[Proposition 2.7(2)]{HNN} there exists a big enough integer $N$ such that $I_{[c\ell+j]}=I_{[cN+j]}I_{[c]}^{\ell-N}$ for all $\ell\ge N$ and all $j\in\{0,\dots,c-1\}$. Setting $\ell-N=k$, we obtain that
			\begin{equation}\label{eq:Ha}
				I_{[c(k+N)+j]}=I_{[cN+j]}I_{[c]}^k,\quad\textup{for all}\ j\in\{0,\dots,c-1\}\ \textup{and all}\ k\ge0.
			\end{equation}
			
			Taking $M=R$ and $I=I_{[c]}$, since $R$ is a domain, Proposition \ref{Prop:known} implies that
			\begin{equation}\label{eq:0}
				I_{[c]}^{k+1}M:_MI_{[c]}\ =\ I_{[c(k+1)]}:_RI_{[c]}\ =\ I_{[ck]}\ =\ I_{[c]}^kM,\quad\textup{for all}\ k\gg0.
			\end{equation}
			
			Next, fix an integer $0<j\le c-1$ and let $M=I_{[cN+j]}$. Then $M$ is a finitely generated $R$-module, and since $R$ is a domain any non-zero element of $I_{[c]}$ is a non-zero divisor on $M$. Hence, Proposition \ref{Prop:known} implies that
			$$
			I_{[c]}^{k+1}M:_{M}I_{[c]}=I_{[c]}^{k}M,\quad\ \textup{for all}\ k\gg0.
			$$
			By (\ref{eq:Ha}) we have $I_{[c]}^{k+1}M=I_{[c(k+N+1)+j]}$ and $I_{[c]}^{k}M=I_{[c(k+N)+j]}$ for all $k\ge0$. Hence
			\begin{equation}\label{eq:2}
				I_{[c]}^{k+1}M:_{M}I_{[c]}=(I_{[c(k+N+1)+j]}:I_{[c]})\cap I_{[cN+j]}=I_{[c(k+N)+j]}
			\end{equation}
			for all $k\gg0$, where the second colon ideal appearing in the above formula is taken in $R$. Using (\ref{eq:0}), for all $k\gg0$ we have
			\begin{equation}\label{eq:1}
				I_{[c(k+N+1)+j]}:I_{[c]}\ \subseteq\ I_{[c(k+N+1)]}:I_{[c]}\ =\ I_{[c(k+N)]}\ \subseteq I_{[cN+j]}.
			\end{equation}
			From this inclusion and equation (\ref{eq:2}) we have $I_{[c(k+1)+j]}:I_{[c]}=I_{[ck+j]}$ for all $k\gg0$. The assertion follows.
		\end{proof}
		
		For a bigraded module $M=\bigoplus_{d,k\ge0}M_{(d,k)}$, we set $M_{(*,k)}=\bigoplus_{d\ge0}M_{(d,k)}$.
		\begin{proof}[Proof of Theorem \ref{Thm:FS24}]
			As in the proof of Lemma \ref{Lem:AssInftyFiltr}, let $c>0$ be the smallest positive integer such that $\mathcal{R}=\mathcal{R(I)}^{(c)}=\bigoplus_{k\ge0}I_{[ck]}$ is standard graded. Lemma \ref{Lem:AssInftyFiltr} implies that for each $j\in\{0,\dots,c-1\}$,
			$$
			\Ass(I_{[ck+j]})\ =\ \Ass(I_{[c(k+1)+j]})
			$$
			for all $k\gg0$. For each $j\in\{0,\dots,c-1\}$, we denote by $\Ass^\infty(\mathcal{I})_j$ the set of those prime ideals $\p\in\Spec(R)$ such that $\p\in\Ass(I_{[ck+j]})$ for all $k\gg0$. Hence, for each $j\in\{0,\dots,c-1\}$, and all $k\gg0$,
			$$
			\v(I_{[ck+j]})\ =\ \min_{\p\in\Ass^\infty(\mathcal{I})_j}\v_\p(I_{[ck+j]}).
			$$
			Thus to prove the theorem, it is enough to show that $\v_\p(I_{[ck+j]})$ is an eventually linear function in $k$, for all $\p\in\Ass^\infty(\mathcal{I})_j$ and all $j\in\{0,\dots,c-1\}$.
			
			Fix $j\in\{0,\dots,c-1\}$. Let $\p\in\Ass^\infty(\mathcal{I})_j$. Then $\p\in\Ass(I_{[ck+j]})$ for all $k\gg0$. Let $\q$ be the product of the elements in the set $X_\p=\{\p_1\in\Ass^\infty(\mathcal{I})_j:\p\subsetneq\p_1\}$ if $X_\p$ is non-empty, otherwise let $\q=R$. Taking into account formula (\ref{eq:Ha}), let $\mathcal{R}'=I_{[cN+j]}\mathcal{R}=\bigoplus_{k\ge0}I_{[c(k+N)+j]}$, that is the extension of $I_{[cN+j]}$ in the Noetherian ring $\mathcal{R}=\mathcal{R}(\mathcal{I})^{(c)}$. Consider the module
			\begin{equation}\label{eq:soc}
				\mathcal{B}\ =\ \frac{(\mathcal{R}'\ :_{\mathcal{R}}\ \p\mathcal{R})}{(\mathcal{R}'\ :_{\mathcal{R}}\ (\p+\q^\infty)\mathcal{R})}.
			\end{equation}
			Notice that $ \mathcal{B}$ is a finitely generated bigraded $\mathcal{R}$-module. By formula (\ref{eq:Ha}), we have $I_{[cN+j]}I_{[ck]}=I_{[c(k+N)+j]}$ for all $k\ge0$. Hence
			$$
			\mathcal{B}\ =\ \bigoplus_{k\ge0}\frac{(I_{[c(k+N+1)+j]}:\pp)\cap I_{[ck]}}{(I_{[c(k+N+1)+j]}:(\p+\q^\infty))\cap I_{[ck]}}.
			$$
			By Lemma \ref{Lem:FS24-2}, we have $(I_{[c(k+1)+j]}:I_{[c]})=I_{[ck+j]}$ for all $k\gg0$. We claim that $I_{[c]}\subseteq\p$. Indeed, $I_{[ck_*+j]}\subseteq\p$ for a big enough integer $k_*$. Since $c(k_*+1)>ck_*+j$ we have $I_{[c]}^{k_*+1}=I_{[c(k_*+1)]}\subseteq I_{[ck_*+j]}\subseteq\p$. Hence $I_{[c]}\subseteq\p$ because $\p$ is prime. It follows that $(I_{[c(k+N+1)+j]}:\pp)\subseteq(I_{[c(k+N+1)]}:I_{[c]})=I_{[c(k+N)]}\subseteq I_{[ck]}$ for all $k\gg0$ and likewise $(I_{[c(k+1)+j]}:(\p+\q^\infty))\subseteq I_{[ck]}$ for all $k\gg0$. Hence
			\begin{equation}\label{eq:k+1Comp}
				\mathcal{B}_{(*,k)}\ =\ \frac{(I_{[c(k+N+1)+j]}:\pp)}{I_{[c(k+N+1)+j]}:(\p+\q^\infty)}.
			\end{equation}
			for all $k\gg0$. Applying Lemma \ref{Lem:Conca} we then deduce that
			\begin{equation}\label{eq:vpalpha}
				\v_\p(I_{[ck+j]})\ =\ \alpha( \mathcal{B}_{(*,k-N-1)})
			\end{equation}
			for all $k\gg0$. Since $ \mathcal{B}_{(*,k-N-1)}$ is a finitely generated $\mathcal{R(I)}^{(c)}$-module and $\mathcal{R(I)}^{(c)}$ is standard graded, equation (\ref{eq:vpalpha}) and variation of \cite[Theorem 8.3.4]{BCRV22} (see also \cite[Proposition 2.5]{FS2}) imply that $\v_\p(I_{[ck+j]})$ is a linear function in $k$ for all $k\gg0$.
		\end{proof}
		
		Keeping the notation of the previous proof, fix an integer $j\in\{0,\dots,c-1\}$. Let $\p\in\Ass^\infty(\mathcal{I})_j$ and let $\mathcal{B}$ as defined in equation (\ref{eq:soc}). It follows from the proof of Theorem \ref{Thm:FS24} that $\mathcal{B}_{(*,k+1)}=I_{[c]}\mathcal{B}_{(*,k)}$ for all $k\gg0$. Hence, we obtain the following consequences of Theorem \ref{Thm:FS24} and equation (\ref{eq:k+1Comp}). The reader may confront this property with \cite[Propositions 2.7 and 3.2]{FS2}.
		\begin{Corollary}
			Under the same assumptions and notation of Theorem \ref{Thm:FS24}, for any $\p\in\Ass^\infty(\mathcal{I})$, there exist $c>0$ such that $(I_{[c(k+1)+j]}:\p)\ =\ I_{[c]}(I_{[ck+j]}:\p)$ for all $k\gg0$ and all $j\in\{0,\dots,c-1\}$. In particular,
			$$
			\lim_{k\rightarrow\infty}\frac{\v_\p(I_{[ck+j]})}{k}\ \in\ \{d:\ (I_{[c]}/\mm I_{[c]})_d\ne0\}.
			$$
		\end{Corollary}
		\begin{proof}
			Keeping the notation as in the proof of Theorem \ref{Thm:FS24}, fix $j\in\{0,\dots,c-1\}$ and let $\mathcal{B}$ be the module defined in equation (\ref{eq:soc}). Then equation (\ref{eq:vpalpha}) holds and so $\v(I_{[ck+j]})=\alpha(\mathcal{B}_{(*,k-N-1)})$ for all $k\gg0$. Since $\mathcal{B}_{(*,k+1)}=I_{[c]}\mathcal{B}_{(*,k)}$ for all $k\gg0$, it follows that the leading coefficient of the asymptotic linear function $\alpha(\mathcal{B}_{(*,k)})$ lies in the set $\{d:(I_{[c]}/\mm I_{[c]})_d\ne0\}$ and this concludes the proof.
		\end{proof}
		
		\begin{Remark}
			\rm In this section, we focused on the $\v$-function $\v(I_{[k]})$, and proved its asymptotic quasi-linearity. This result can be generalized, with the obvious modifications, to show that the $\v$-functions $\v(I_{[k]}M/I_{[k+1]}N)$ and $\v(M/I_{[k+1]}N)$ are quasi-linear for $k\gg0$, where $\mathcal{I}=\{I_{[k]}\}_{k\ge0}$ is a Noetherian graded filtration, $M$ is a finitely generated graded $R$-module, and $N$ is a graded submodule of $M$. See also the papers \cite{Fior24} and \cite{Ghosh24}.
		\end{Remark}
		
		\section{When $\p\in\Ass^\infty(\mathcal{I})$?}\label{Sec2-FS24}
		
		Let $\mathcal{I}=\{I_{[k]}\}_{k\ge0}$ be a Noetherian graded filtration of $R$. Given $\p\in\Spec(R)$, when do we have that $\p\in\Ass^\infty(\mathcal{I})$? The next theorem answers this question. For the proof of this result, we recall some basic facts.\smallskip
		
		Let $(R,\m,K)$ be a Noetherian local ring, $I\subset R$ be an ideal, and $M$ be a finitely generated $R$-module. Then, all minimal system of generators of $I$ have the same lenght, namely $\mu(I)=\dim_K (I/\m I)$. Let $\mu(I)=m$, ${\bf f}:f_1,\dots,f_m$ be a minimal system of generators of $I$, and let $H_i({\bf f};M)$ be the $i$th Koszul homology of ${\bf f}$ with respect to $M$. Then \cite[Proposition 1.6.21]{BH} implies that if ${\bf g}:g_1,\dots,g_m$ is another minimal system of generators of $I$, we have $H_i({\bf f};M)\cong H_i({\bf g};M)$ for all $i$. In other words, the Koszul homology modules do not depend upon the choice of a minimal system of generators of $I$. Therefore, we write $H_i(I;M)$ to denote $H_i({\bf f};M)$, where ${\bf f}$ can be any minimal system of generators of $I$.
		\begin{Theorem}\label{Thm:pinAss8(I)}
			Under the assumptions and notation of Theorem \ref{Thm:FS24}, for a prime ideal $\p\in\Spec(R)$, the following conditions are equivalent:
			\begin{enumerate}
				\item[\textup{(a)}] $\p\in\Ass^\infty(\mathcal{I})$.
				\item[\textup{(b)}] $\p\in\Ass(I_{[k]})$ for infinitely many $k$.
				\item[\textup{(c)}] The Krull dimension $\dim H_{\mu(\p R_\p)-1}(\p R_\p\,;\,\mathcal{R}(\mathcal{I})_\p)$ is positive.
			\end{enumerate}
		\end{Theorem}
		\begin{proof}
			That (a) and (b) are equivalent is by definition of the set $\Ass^\infty(\mathcal{I})$.
			
			Before proving that (b) and (c) are equivalent, we notice the following fact. By localizing at $\p$, we see that $\p\in\Ass(I_{[k]})$ if and only if $\p R_\p\in\Ass((I_{[k]})_\p)$. Since $R_\p$ is again a Noetherian ring, then $\p R_\p\in\Ass((I_{[k]})_\p)$ if and only if there exists $f\in R_\p$ such that $((I_{[k]})_\p:f)=\p R_\p$. This implies that $((I_{[k]})_\p:\p R_\p)/(I_{[k]})_\p\ne0$. Conversely, if $((I_{[k]})_\p:\p R_\p)/(I_{[k]})_\p\ne0$ and $g+(I_{[k]})_\p$ is a non-zero element belonging to this module, then $((I_{[k]})_\p:g)$ is a proper ideal that contains $\p R_\p$. Since $\p R_\p$ is the unique maximal ideal of $R_\p$, we see that $((I_{[k]})_\p:g)=\p R_\p$ and so $\p R_\p\in\Ass((I_{[k]})_\p)$.
			
			Summarizing, $\p\in\Ass(I_{[k]})$ if and only if $((I_{[k]})_\p:\p R_\p)/(I_{[k]})_\p\ne0$.
			
			Now we prove that (b) and (c) are equivalent. By the previous fact, $\p\in\Ass(I_{[k]})$ for infinitely many $k$ if and only if $((I_{[k]})_\p:\p R_\p)/(I_{[k]})_\p\ne0$ for infinitely many $k$.
			
			Notice that $\mathcal{R(I)}_\p$ is a Noetherian graded ring. Indeed, localization preserves the Noetherian property. Whereas, $\mathcal{R(I)}_\p=\bigoplus_{k\ge0}(I_{[k]})_\p$ is graded, because applying localization to the property (iii) of $\mathcal{I}$ yields $(I_{[k]})_\p(I_{[\ell]})_\p\subseteq(I_{[k+\ell]})_\p$ for all $k,\ell\ge0$. Hence, the Koszul homology $H_{\mu(\p R_\p)-1}(\p R_\p;\mathcal{R}(\mathcal{I})_\p)$ is a graded $\mathcal{R}(\mathcal{I})_\p$-module whose $k$th graded component is $H_{\mu(\p R_\p)-1}(\p R_\p;(I_{[k]})_\p)=((I_{[k]})_\p:\p R_\p)/(I_{[k]})_\p$. Thus (b) holds if and only infinitely many graded components of $H_{\mu(\p R_\p)-1}(\p R_\p;\mathcal{R}(\mathcal{I})_\p)$ are non-zero. But this is clearly equivalent to condition (c).
		\end{proof}
		
		In the case of the ordinary powers of monomial ideals of a standard graded polynomial ring with coefficients over a field, this result was firstly noted by Bayati, Herzog and Rinaldo \cite{BHR12}. Furthermore, in such a situation one can replace ordinary localization with monomial localization. See \cite[Section 2]{BHR12} for the details.
		
		\section{Integer Programming and the $\v$-number}\label{Sec3-FS24}
		
		In this section, we show how the computation of the $\v$-function of certain graded filtrations of monomial ideals is related to integer programming problems.\smallskip
		
		Let $k$ be an integer and let ${\bf a}=(a_1,\dots,a_n),{\bf b}=(b_1,\dots,b_n)\in\ZZ^n$ be integral vectors. We set $k{\bf a}=(ka_1,\dots,ka_n)$ and ${\bf a+b}=(a_1+b_1,\dots,a_n+b_n)$. The modulus of ${\bf a}$ is defined as $|{\bf a}|=a_1+\dots+a_n$. We denote by ${\bf a}[i]=a_i$ the $i$th component of ${\bf a}$. Let ${\bf e}_1,\dots,{\bf e}_n$ be the canonical basis of $\ZZ^n$. That is, ${\bf e}_i[j]=1$ for $i=j$ and ${\bf e}_i[j]=0$ if $i\ne j$. We denote by ${\bf 0}$ the vector $(0,0,\dots,0)\in\ZZ^n$.\medskip
		
		Let ${\bf A}=\{{\bf a}_1,\dots,{\bf a}_m\}$ and ${\bf B}=\{{\bf b}_1,\dots,{\bf b}_\ell\}\subseteq\{{\bf e}_1,\dots,{\bf e}_n\}$ be two finite collections of integral vectors of $\ZZ^n_{\ge0}$. Let $k\ge1$. We define $\mathcal{P}_k({\bf A})$ to be the following set of lattice points of $\ZZ^n$:
		$$
		\mathcal{P}_k({\bf A})=\mathcal{P}_k({\bf a}_1,\dots,{\bf a}_m)=\big\{\!\sum_{i=1}^mk_i{\bf a}_i+\sum_{i=1}^nh_i{\bf e}_i\ :\ k_i,h_i\ge0,\ \sum_{i=1}^mk_i\ge k\big\}.
		$$
		
		For any $k\ge1$, we consider the following integer program:
		$$
		P_{{\bf A,B},k}\ :\ \ \begin{cases}
			\textup{minimize}&{\bf c}\in\ZZ^n,\\
			\textup{subject to}&{\bf c}+{\bf d}\in\mathcal{P}_k({\bf A})\ \textup{if and only if}\ {\bf d}\in\mathcal{P}_1({\bf B}).
		\end{cases}
		$$
		
		We say that a solution ${\bf c}$ of $P_{{\bf A,B},k}$, if it exists, is an \textit{optimal solution} if its module $|{\bf c}|$ is minimum among all solutions of the problem.\medskip
		
		Let $K$ be a field and $S=K[x_1,\dots,x_n]$. For a non-empty subset $A$ of $\ZZ_{\ge0}^n$, let $I({\bf A})\subset S$ be the monomial ideal generated by ${\bf x^a}=x_1^{a_1}\cdots x_n^{a_n}$, ${\bf a}=(a_1,\dots,a_n)\in{\bf A}$. Then $I({\bf A})^k=I(\mathcal{P}_k({\bf A}))$ for all $k\ge1$. Notice that $I({\bf A})\ne S$ if and only if ${\bf 0}\notin{\bf A}$.\medskip
		
		By Brodmann \cite{B79}, if $I\subset S$ is a homogeneous ideal, the sets $\Ass(I^k)$ stabilize, that is: $\Ass(I^k)=\Ass(I^{k+1})$ for all $k\gg0$. We denote by $\Ass^\infty(I)$ the common sets $\Ass(I^k)$ for $k\gg0$.
		
		\begin{Theorem}\label{Thm:LP-v-Number}
			Let ${\bf A},{\bf B}\subset\ZZ_{\ge0}^n$ be non-empty finite subsets such that ${\bf0}\notin{\bf A}$ and ${\bf B}\subseteq\{{\bf e}_1,\dots,{\bf e}_n\}$. Then, the following conditions are equivalent.
			\begin{enumerate}
				\item[\textup{(a)}] The integer program $P_{{\bf A,B},k}$ has a solution for all $k\gg0$.
				\item[\textup{(b)}] $I({\bf B})=(x_i:{\bf e}_i\in B)\in\Ass(I({\bf A})^k)$ for all $k\gg0$.
				\item[\textup{(c)}] The Krull dimension of the module $H_{\mu(\p R_\p)-1}(\p R_\p\,;\,\mathcal{R}(\mathcal{I})_\p)$ is positive, where $\mathcal{I}=\{I({\bf A})^k\}_{k\ge0}$ and $\p=I({\bf B})$.
			\end{enumerate}
			Furthermore, if any of the equivalent conditions hold, then for all $k\gg0$ the modulus of the optimal solution of $P_{{\bf A,B},k}$ is a linear function of the form $ak+b$, where $a,b\in\ZZ$ are suitable integers. In particular, $a\in\{|{\bf a}|\ :\ {\bf a}\in{\bf A}\}$.
		\end{Theorem}
		\begin{proof}
			Fix a field $K$ and an integer $k\ge1$. Then, there exists a solution ${\bf c}\in\ZZ_{\ge0}^n$ of the integer program $P_{{\bf A,B},k}$ if and only if we have ${\bf x^c}{\bf x^d}\in I(\mathcal{P}_k({\bf A}))=I({\bf A})^k$ only when ${\bf x^d}\in I({\bf B})$. That is, there exists a solution ${\bf c}\in\ZZ_{\ge0}^n$ of the integer program $P_{{\bf A,B},k}$ if and only if $(I({\bf A})^k:{\bf x^c})=I({\bf B})$. Thus, condition (a) is equivalent to condition (b). Let $\mathcal{I}=\{I({\bf A})^k\}_{k\ge0}$. By Theorem \ref{Thm:pinAss8(I)}, (b) is equivalent to the fact that $H_{\mu(\p R_\p)-1}(\p R_\p\,;\,\mathcal{R}(\mathcal{I})_\p)$ is non-zero for infinitely many $k$, where $\p=I({\bf B})$. This is equivalent to condition (c).
			
			Finally, suppose that (a) or (b) or (c) holds. For $k\gg0$ the modulus $|{\bf c}|$ of the optimal solution ${\bf c}$ of $P_{{\bf A,B},k}$ is the $\v_{I({\bf B})}$-number $\v_{I({\bf B})}(I({\bf A})^k)$. By Theorem \ref{Thm:FS24}, (see also \cite[Theorem 3.1]{FS2} or \cite[Theorem 1.1]{Conca23}) $\v_{I({\bf B}))}(I({\bf A})^k)=ak+b$ for all $k\gg0$. By \cite[Remark 1.3]{Conca23}, (see also \cite[Proposition 3.2]{FS2}) we have $a\in\{\deg({\bf x^a}):{\bf x^a}\in\mathcal{G}(I({\bf A}))\}\subseteq\{|{\bf a}|\ :\ {\bf a}\in{\bf A}\}$, as desired.
		\end{proof}
		
		As before, let ${\bf A}=\{{\bf a}_1,\dots,{\bf a}_m\}$ and ${\bf B}=\{{\bf b}_1,\dots,{\bf b}_\ell\}\subseteq\{{\bf e}_1,\dots,{\bf e}_n\}$ be two finite collections of integral vectors of $\ZZ^n_{\ge0}$. Let $k\ge1$. We define $\overline{\mathcal{P}}_k({\bf A})$ to be the set of lattice points of the convex hull in $\RR^n$ of the set $\mathcal{P}({\bf A})$. That is, we set $\overline{\mathcal{P}}_k({\bf A})\ =\ \textup{conv}(\mathcal{P}_k({\bf A}))\cap\ZZ_{\ge0}^n$.
		By \cite[Corollary 1.4.3]{HH2011}, $I(\overline{\mathcal{P}}_k({\bf A}))$ is equal to the integral closure of $\overline{I(\mathcal{P}_k({\bf A}))}$. Hence $I(\overline{\mathcal{P}}_k({\bf A}))=\overline{I({\bf A})^k}$ for all $k\ge1$.
		
		For any ${\bf A}=\{{\bf a}_1,\dots,{\bf a}_m\}$, ${\bf B}=\{{\bf b}_1,\dots,{\bf b}_\ell\}\subseteq\{{\bf e}_1,\dots,{\bf e}_n\}$ and $k\ge1$, we consider the following integer program:
		$$
		\overline{P}_{{\bf A,B},k}\ :\ \ \begin{cases}
			\textup{minimize}&{\bf c}\in\ZZ^n,\\
			\textup{subject to}&{\bf c}+{\bf d}\in\overline{\mathcal{P}}_k({\bf A})\ \textup{if and only if}\ {\bf d}\in\mathcal{P}_1({\bf B}).
		\end{cases}
		$$
		
		It is well-known that $\mathcal{I}=\{\overline{I^k}\}_{k\ge0}$ is a Noetherian graded filtration, if $I\subset S$ is a homogeneous ideal. By Brodmann \cite{B79}, or \cite{MAE79}, the sets $\Ass(\overline{I^k})$ stabilize: that is $$\Ass(\overline{I^k})=\Ass(\overline{I^{k+1}})$$ for all $k\gg0$. We set $$\Ass^\infty(\{\overline{I^k}\}_{k\ge0})=\overline{\textup{Ass}}^\infty(I).$$ As in Theorem \ref{Thm:LP-v-Number} one can show that
		\begin{Theorem}\label{Thm:LP-v-Number1}
			Let ${\bf A},{\bf B}\subset\ZZ_{\ge0}^n$ be non-empty finite subsets such that ${\bf0}\notin{\bf A}$ and ${\bf B}\subseteq\{{\bf e}_1,\dots,{\bf e}_n\}$. Then, the following conditions are equivalent.
			\begin{enumerate}
				\item[\textup{(a)}] The integer program $\overline{P}_{{\bf A,B},k}$ has a solution for all $k\gg0$.
				\item[\textup{(b)}] $I({\bf B})=(x_i:{\bf e}_i\in B)\in\overline{\Ass}^\infty(I({\bf A}))$.
				\item[\textup{(c)}] The Krull dimension of the module $H_{\mu(\p R_\p)-1}(\p R_\p\,;\,\mathcal{R}(\mathcal{I})_\p)$ is positive, where $\mathcal{I}=\{\overline{I({\bf A})^k}\}_{k\ge0}$ and $\p=I({\bf B})$.
			\end{enumerate}
			Furthermore, if any of the equivalent conditions hold, then for all $k\gg0$ the modulus of the optimal solution of $\overline{P}_{{\bf A,B},k}$ is a quasi-linear function.
		\end{Theorem}\medskip
		
		\section{The package \texttt{VNumber}}\label{Sec4-FS24}
		
		In this section we give a self contained introduction to the \textit{Macaulay2} \cite{GDS} package \texttt{VNumber} \cite{FSPack} and demonstrate its main features.\medskip
		
		Let $I\subset S=K[x_1,\dots,x_n]$ be a homogeneous ideal, and let $\p\subset S$ be a prime ideal. We set $\Ass^\infty(I)=\Ass^\infty(\{I^k\}_{k\ge0})$ and by $\Max^\infty(I)$ we denote the set of primes of $\Ass^\infty(I)$ which are maximal with respect to the inclusion.
		
		In the next table we collect the functions available in the package. Except for the functions \texttt{reesMap}$(I)$, \texttt{vNumberP}$(I,\p)$ and \texttt{vNumber}$(I)$, all other functions require that $I$ is a monomial ideal.
		\small\begin{table}[H]
			\centering
			\begin{tabular}{ll}
				\rowcolor{black!20}\bottomrule[1.05pt]
				Functions&Description\\
				\toprule[1.05pt]
				\texttt{reesMap}$(I)$&Computes the Rees map of $I$\\
				\texttt{stablePrimes}$(I)$&Computes $\Ass^\infty(I)$\\
				\texttt{stableMax}$(I)$&Computes $\Max^\infty(I)$\\
				\texttt{isStablePrime}$(I,\p)$&Checks if $\p\in\Ass^\infty(I)$\\
				\texttt{vNumberP}$(I,\p)$&Computes $\v_\p(I)$\\
				\texttt{vNumber}$(I)$&Computes $\v(I)$\\
				\texttt{soc}$(I,\p)$&Computes $H_{\mu(\p R_\p)-1}(\p R_\p\,;\,\mathcal{R}(\mathcal{I})_\p)$\\
				\texttt{vFunctionP}$(I,\p)$&Computes $\v_\p(I^k)=a_\p k+b_\p$ for $k\gg0$\\
				\texttt{vFunction}$(I)$&Computes $\v(I^k)=ak+b$ for $k\gg0$\\
				\bottomrule[1.05pt]
			\end{tabular}\medskip
			\caption{List of the functions of \texttt{VNumber}.}
		\end{table}
		\normalsize
		
		Let $I\subset S$ be a homogeneous ideal and let $f_1,\dots,f_m$ be any minimal homogeneous generating set of $I$. Then $\mathcal{R}(I)=\bigoplus_{k\ge0}I^k$ may be computed as follows. First, to keep track of the natural bidegree of $\mathcal{R}(I)$ we introduce the variable $t$ and we replace $\mathcal{R}(I)$ by $\bigoplus_{k\ge0}I^kt^k$. Then, we can consider the bigraded map
		$$
		\varphi_I\ :\ T=K[x_1,\dots,x_n,y_1,\dots,y_m]\rightarrow S[t]=K[x_1,\dots,x_n,t]
		$$
		defined by setting $\varphi(x_i)=x_i$ for $1\le i\le n$ and $\varphi(y_j)=f_jt$ for $1\le j\le m$, and with $\bideg(x_i)=(1,0)$ and $\bideg(y_j)=(\deg f_j,1)$. We call $\varphi_I$ the \textit{Rees map} of $I$ (induced by $f_1,\dots,f_m$). Notice that $\varphi_I$ depends on the choice of a minimal homogeneous generating set of $I$. However, no matter which minimal system of homogeneous generators of $I$ we choose, we always have
		$$
		\mathcal{R}(I)\ \cong\ T/\ker\varphi_I.
		$$
		
		The function \texttt{reesMap}$(I)$ allows to compute the map $\varphi_I$, for some choice of the minimal homogeneous generating set of $I$.
		
		The functions \texttt{stablePrimes}$(I)$, \texttt{stableMax}$(I)$, \texttt{isStablePrime}$(I,\p)$ are based on Theorem \ref{Thm:pinAss8(I)} which says that $\p\in\Ass^\infty(I)$ if and only if the Krull dimension of the module $H_{\mu(\p R_\p)-1}(\p R_\p\,;\,\mathcal{R}(I)_\p)$ is positive. At the moment, for the function \texttt{isStablePrime}$(I,\p)$ it is required that $I$ is a monomial ideal, so that one can replace ordinary localization by monomial localization as shown in \cite[Section 2]{BHR12}. Indeed our code for this function is just the one given in \cite{BHR11}. On the other hand, if $I$ is a homogeneous ideal (not necessarily monomial) while \textit{Macaulay2} \cite{GDS} can compute the module $H_{\mu(\p R_\p)-1}(\p R_\p\,;\,\mathcal{R}(I)_\p)$, in most cases it is not able to compute its Krull dimension.
		
		To compute \texttt{vNumberP}$(I,\p)$ and \texttt{vNumber}$(I)$ the Lemma \ref{Lem:Conca} of Conca is used. In this case $I$ can be any homogeneous ideal of $S$. When $I$ is a monomial ideal, the function \texttt{soc}$(I,\p)$ computes the module $H_{\mu(\p R_\p)-1}(\p R_\p\,;\,\mathcal{R}(\mathcal{I})_\p)$ and the function \texttt{reesMap}$(I)$. The computation of the $\v$-function \texttt{vFunction}$(I)$ and of the $\v_\p$-function \texttt{vFunctionP}$(I,\p)$ employs the proof of \cite[Theorem 1.2]{Conca23} or also Theorem \ref{Thm:FS24}.\bigskip
		
		The next example illustrates how to use the package.\bigskip
		
		\small
		\verb|i1: loadPackage "VNumber"|\smallskip
		
		\verb|i2: S = QQ[x_1..x_6];|\smallskip
		
		\verb|i3: I = ideal(x_1*x_2,x_1*x_3,x_2*x_3,x_2*x_4,x_3*x_4,x_4*x_5,x_5*x_6);|\smallskip
		
		\verb|i4: stablePrimes I|\smallskip
		
		\verb|o4:| $\{\texttt{ideal}{}\left(x_{5},\,x_{3},\,x_{2}\right),\texttt{ideal}{}\left(x_{5},\,x_{4},\,x_{2},\,x_{1}\right),\texttt{ideal}{}\left(x_{5},\,x_{4},\,x_{3},\,x_{1}\right),$
		
		\verb|   | $\texttt{ideal}{}\left(x_{6},\,x_{4},\,x_{2},\,x_{1}\right),\texttt{ideal}{}\left(x_{6},\,x_{4},\,x_{3},\,x_{1}\right),\texttt{ideal}{}\left(x_{6},\,x_{4},\,x_{3},\,x_{2}\right),$
		
		\verb|   | $\texttt{ideal}{}\left(x_{1},\,x_{2},\,x_{3},\,x_{4},\,x_{6}\right),\texttt{ideal}{}\left(x_{1},\,x_{2},\,x_{3},\,x_{4},\,x_{5}\right),\texttt{ideal}{}\left(x_{1},\,x_{2},\,x_{3},\,x_{4},\,x_{5},\,x_{6}\right)\}$\smallskip
		
		\verb|i5: P = (stablePrimes I)#1;|\smallskip
		
		\verb|i6: vFunctionP(I,P)|\smallskip
		
		\verb|o6: {2,0}|\smallskip
		
		\verb|i7: vFunction I|\smallskip
		
		\verb|o7: {2,-1}|\bigskip
		
		\normalsize
		
		\section{Open questions}\label{Sec5-FS24}
		Let $\mathcal{I}=\{I_{[k]}\}_{k\ge0}$ be a Noetherian graded filtration of $R$.
		\begin{Question}\label{FSQuest1}
			Is it true that $\v(I_{[k]})<\reg(I_{[k]})$ for all $k\gg0$?
		\end{Question}
		\begin{Problem}
			Find an example of non Noetherian graded filtration $\mathcal{I}=\{I_{[k]}\}_{k\ge0}$ of $R$ whose $\v$-function $\v(I_{[k]})$ is not an eventually quasi-linear function in $k$. Does the limit $\lim_{k\rightarrow\infty}\v(I_{[k]})/k$ exists? If so, what it is equal to?
		\end{Problem}
		
		This problem was recently addressed in \cite{VS24} where an example of a non Noetherian graded filtration $\mathcal{I}$ whose $\v$-function $\v(I_{[k]})$ is not an eventually quasi-linear function is provided.
		
		Suppose now that $R$ is the standard graded polynomial ring $S=K[x_1,\dots,x_n]$ over an infinite field $K$. Let $I\subset S$ be a homogeneous ideal, and let $\mathcal{I}=\{I^k\}_{k\ge0}$ be the $I$-adic filtration. Let $\p\in\Ass^\infty(I)$. While we know that $\v_\p(I^k)$ and $\v(I^k)$ are eventually linear functions \cite[Theorem 3.1]{FS2}, their initial behaviour is quite mysterious. In particular, we do not know if they can have any given number of strict local maxima. Due to experimental evidence, we ask the following
		\begin{Question}\phantom{a}
			\begin{enumerate}
				\item[\textup{(a)}] Is it true that $\v_\p(I^k)<\v_\p(I^{k+1})$ for all $\p\in\Ass^\infty(I)$ and all $k\ge1$?
				\item[\textup{(b)}] Is it true that $\v(I^k)<\v(I^{k+1})$ for all $k\ge1$?
			\end{enumerate}
		\end{Question}
		
		In \cite[Theorem 4.1]{FS2} we proved that $\lim_{k\rightarrow\infty}\v(I^k)/k=\alpha(I)$. Thus $\v(I^k)=\alpha(I)k+b$ for all $k\gg0$ and a certain integer $b\in\ZZ$. Since $\v(I^k)=\max_{\p\in\Ass^\infty(I)}\v_\p(I^k)$ for all $k\gg0$, and each $\v_\p(I^k)$ is an eventually linear function in $k$, there exists at least one prime $\p\in\Ass^\infty(I)$ such that $\lim_{k\rightarrow\infty}\v_\p(I^k)/k=\alpha(I)$. In view of this observation and some experimental evidence, we expect that
		\begin{Conjecture}
			Let $I\subset S$ be a homogeneous ideal. For all $\p\in\Max^\infty(I)$, we have
			$$
			\lim_{k\rightarrow\infty}\frac{\v_\p(I^k)}{k}\ =\ \alpha(I).
			$$
		\end{Conjecture}
		
		Assume that $I\subset S$ is a proper monomial ideal. By \cite[Proposition 2.2]{F2023} for all $k\gg0$, $\v(I^k)=ak+b$ where $a=\alpha(I)\ge1$ and $b\ge-1$. By \cite[Theorem 5.5]{DMNB23}, $\reg(I^k)=ck+d$ where $c\ge1$ and $d\ge0$. In view of Question \ref{FSQuest1} we ask the following
		\begin{Question}
			Let $a,b,c,d$ integers such that $a\ge1$, $b\ge-1$, $c\ge1$, $d\ge0$ and $ak+b\le ck+d$ for all $k\gg0$. Can we find a monomial ideal $I$ in some polynomial ring $S$ such that $\v(I^k)=ak+d$ and $\reg(I^k)=ck+d$ for all $k\gg0$.
		\end{Question}
		
		In \cite{BMS24}, the authors introduced the \textit{$\v$-stability index} of a graded ideal $I\subset S$ as
		$$
		\textup{vstab}(I)\ =\ \min\{t\ :\ \v(I^k)=\alpha(I)k+b,\ \textit{for all}\ k\ge t,\ \textit{and some}\ b\in\ZZ\}.
		$$
		Let $\p\in\Ass^\infty(I)$. We define the \textit{$\v_\p$-stability index} of $I\subset S$ as the integer
		$$
		\v_\p\!\textup{-stab}(I)\ =\ \min\{t\ :\ \v_\p(I^k)=ak+b,\ \textit{for all}\ k\ge t,\ \textit{and some}\ a,b\in\ZZ\}.
		$$
		
		Recall that by $\textup{astab}(I)$ one denotes the integer
		$$
		\textup{astab}(I)\ =\ \min\{t\ :\ \Ass(I^k)=\Ass^\infty(I),\ \textit{for all}\ k\ge t\}.
		$$
		Whereas, for $\p\in\Ass^\infty(I)$ we define the integer
		$$
		\textup{astab}_\p(I)\ =\ \min\{t\ :\ \p\in\Ass(I^k),\ \textit{for all}\ k\ge t\}.
		$$
		
		Notice that $\max_{\p\in\Ass^\infty(I)}\textup{astab}_\p(I)=\textup{astab}(I)$ and $\textup{astab}_\p(I)\le\v_\p\!\textup{-stab}(I)$. On the other hand,
		$$
		\textup{vstab}(I)\ \le\ \max_{\p\in\Ass^\infty(I)}\v_\p\!\textup{-stab}(I)
		$$
		and the inequality may be strict.
		
		For instance, for any edge ideal $I(G)$ having linear resolution, $\textup{vstab}(I(G))=1$. Indeed, by \cite[Theorem 5.1]{F2023} we have $\v(I(G)^k)=2k-1$ for all $k\ge1$. However, in most cases $\Ass(I(G))\ne\Ass^\infty(I(G))$. Hence, for some $\p\in\Ass^\infty(I(G))$, $\v_\p\!\textup{-stab}(I(G))\ge\textup{astab}_\p(I(G))\ge2>1=\textup{vstab}(I(G))$. For a simple example, consider $I(G)=(x_1x_2,x_1x_3,x_2x_3)$, the edge ideal of the complete graph on three vertices. Notice that $I(G)$ has linear resolution. The package \texttt{VNumber} \cite{FSPack} reveals that $\Ass(I(G))=\{(x_1,x_2),(x_1,x_3),(x_2,x_3)\}$ and $\Ass(I(G)^k)=\Ass(I(G))\cup\{(x_1,x_2,x_3)\}$ for all $k\ge2$. So that $\v_{(x_1,x_2,x_3)}\!\textup{-stab}(I(G))=2$.\smallskip
		
		This example also shows that $\textup{vstab}(I)$ can be smaller than $\textup{astab}(I)$. On the other hand, if $I$ is a monomial ideal of $K[x,y]$, by \cite[Corollary 5.3]{FS2} we have $\textup{astab}(I)=1$, while $\textup{vstab}(I)\ge 1$ can be bigger than $\textup{astab}(I)$. For instance, consider $I=(x^{11},x^7y^5,x^4y^{10},xy^{11},y^{14})\subset S=K[x,y]$. Then $\textup{astab}(I)=1$ but $\textup{vstab}(I)=2$. Thus, there is no comparison between the stability indices $\textup{astab}(I)$ and $\textup{vstab}(I)$.\smallskip
		
		Again, let $I\subset S$ be a graded ideal. By Brodmann \cite{B79a}, $\depth I^k=\depth I^{k+1}$ for all $k\gg0$. We denote by $\textup{dstab}(I)$ the least integer $t$ such that $\depth I^k=\depth I^{k+1}$ for all $k\ge t$. Finally, let $\reg(I^k)=ck+d$ for all $k\gg0$. We denote by $\textup{rstab}(I)$ the least integer $t$ such that $\reg(I^k)=ck+d$ for all $k\ge t$.
		
		\begin{Question}
			Given positive integer $a,d,r,v\ge 1$, can we find a graded ideal $I$ in some polynomial ring $S$ such that
			$$
			\textup{astab}(I)=a,\ \ \ \textup{dstab}(I)=d,\ \ \ \textup{rstab}(I)=r,\ \ \ \textup{vstab}(I)=v?
			$$
		\end{Question}
		
		Let $I\subset S$ be a monomial ideal. In \cite{Hoa06}, Le Tuan Hoa bounded the index $\textup{astab}(I)$ of $I$ using exquisite techniques from combinatorics and asymptotic linearity of integer programming. For similar recent results, see the papers \cite{HRR23,Hoa22}.
		\begin{Problem}
			Let $I\subset S$ be a monomial ideal. Using integer programming techniques, determine good upper bounds for the $\v$-stability index $\textup{vstab}(I)$ of $I$.
		\end{Problem}
		
		\textbf{Acknowledgment.} We gratefully acknowledge the referee for their suggestions and comments that greatly improved the quality of the manuscript. We thank A. Vanmathi and P. Sarkar for their careful reading of the paper and pointing out some inaccuracies in a previous version of the manuscript.
		
		A. Ficarra was partly supported by the Grant JDC2023-051705-I funded by
		MICIU/AEI/10.13039/501100011033 and by the FSE+.
		
	\end{document}